\documentclass[12pt]{article}
\usepackage{amsmath, amssymb, amsfonts, latexsym, amsthm, amscd, epsfig, graphics, subfig}
\usepackage{hyperref}
\usepackage{lineno}

\newtheorem{theorem}{Theorem}[section]
\newtheorem{definition}[theorem]{Definition}
\newtheorem{lemma}[theorem]{Lemma}
\newtheorem{proposition}[theorem]{Proposition}

\newtheorem{remark}[theorem]{Remark}
\numberwithin{equation}{section}
\allowdisplaybreaks[4]

\newcommand*\patchAmsMathEnvironmentForLineno[1]{%
  \expandafter\let\csname old#1\expandafter\endcsname\csname #1\endcsname
  \expandafter\let\csname oldend#1\expandafter\endcsname\csname end#1\endcsname
  \renewenvironment{#1}%
     {\linenomath\csname old#1\endcsname}%
     {\csname oldend#1\endcsname\endlinenomath}}%
\newcommand*\patchBothAmsMathEnvironmentsForLineno[1]{%
  \patchAmsMathEnvironmentForLineno{#1}%
  \patchAmsMathEnvironmentForLineno{#1*}}%
\AtBeginDocument{%
\patchBothAmsMathEnvironmentsForLineno{equation}%
\patchBothAmsMathEnvironmentsForLineno{align}%
\patchBothAmsMathEnvironmentsForLineno{flalign}%
\patchBothAmsMathEnvironmentsForLineno{alignat}%
\patchBothAmsMathEnvironmentsForLineno{gather}%
\patchBothAmsMathEnvironmentsForLineno{multline}%
}

\title{\bf Deforming Locally Convex Curves into Curves of Constant $k$-order Width}

\author{\ {\bf Laiyuan Gao ~~~~Horst Martini ~~~~Deyan Zhang\thanks{The corresponding author.}} }
\date{}

\begin{document}
\maketitle

\noindent {\bf Abstract} A nonlocal curvature flow is introduced to evolve locally convex curves in the plane.
It is proved that this flow with any initial locally convex curve has a global solution, keeping the local convexity and the elastic energy of the evolving curve,
and that, as the time goes to infinity, the curve converges to a smooth, locally convex curve of constant $k$-order width.
In particular, the limiting curve is a multiple circle if and only if the initial locally convex curve is $k$-symmetric.\\



\baselineskip 15pt

\section{Introduction}\label{s1}
\setcounter{equation}{0}
Let $X_0: S^1 \rightarrow \mathbb{E}^2$ be a $C^2$, immersed and closed curve in the Euclidean plane. If its relative curvature $\kappa$ is positive everywhere,
then $X_0$ is called a \emph{locally convex curve}. If $X_0$ is also embedded, then it is called a \emph{convex curve}.

In this paper a new curvature flow is established to evolve locally convex curves into curves of constant $k$-order width. This work is motivated by the following series of studies.
Let $X: S^1\times [0, \omega) \rightarrow \mathbb{E}^2$ be a family of smooth and locally convex curves in the plane, with $s$ and $\theta$ denoting the arc length parameter
and the tangent angle, respectively. Since $\frac{d\theta}{ds}$ equals the curvature $\kappa(s)>0$ for all $s$, the angle $\theta$ can be used as a parameter. For every $\theta$,
$p(\theta, t)= -\langle X(\theta, t), N(\theta, t)\rangle$ is called the value of the \emph{support function}, where $N(\theta, t)$ is the unit normal.
Gao and Pan studied in \cite{Gao2014} a curvature flow for convex curves given by
\begin{equation}\label{eq:1.1.202302}
\left\{
\begin{aligned}
&\frac{\partial X}{\partial t}(\theta, t)=\left(w(\theta, t)-\eta(\theta, t)\right)N(\theta, t),
\\
&X(\theta, 0)=X_0(\theta), \quad  (\theta, t)\in [0, 2\pi] \times [0, T_{\max}),
\end{aligned}
\right.
\end{equation}
where $w(\theta, t)=p(\theta, t)+p(\theta+\pi, t)$ is the width function; $\eta(\theta, t) = \rho(\theta, t)+\rho(\theta+\pi, t)$ and $\rho(\theta, t) = \frac{1}{\kappa(\theta, t)}$
is the radius of curvature.
They proved that this flow drives the evolving curve to a limiting convex curve of constant width, if the
initial curve satisfies a $1/3$ curvature pinching condition. Later, this result was generalized by Gao and Zhang \cite{Gao-Zhang-2017} for the evolution
of convex hypersurfaces in higher dimensional Euclidean space.
Another generalized model was presented by Fang in the paper \cite{Fang2017}. He replaced $w$ and $\eta$ in the equation (\ref{eq:1.1.202302}) by the \emph{$k$-order width function} $w_k(\theta)=p(\theta)+p\left(\theta+\frac{2\pi}{k}\right)+\cdots+p\left(\theta+\frac{2(k-1)\pi}{k}\right)$
and $\eta_k = \rho(\theta)+ \rho\left(\theta+\frac{2\pi}{k}\right)+\cdots+ \rho\left(\theta+\frac{2(k-1)\pi}{k}\right)$, respectively,
where $k\geq 2$ is a positive integer. He proved that under a $\frac{2k-1}{k-1}$ curvature pinching condition the curvature flow deforms an initial convex curve
into a limiting curve of constant $k$-order width.

To guarantee the global existence for the above curvature flows, some curvature pinching condition of the initial curve or hypersurface is needed. So a natural
question is whether one can construct a proper curvature flow which evolves every initial curve globally and drives the evolving curve into the
limiting curve. To settle this problem, we consider in this paper a new curvature flow of locally convex curves.
Let $X_0$ be a smooth, closed and locally convex planar curve parameterized by the tangent angle $\theta$.
Denote by $m$ the \emph{winding number} of $X_0$. It equals the total curvature divided by $2\pi$, i.e., $m= \frac{1}{2\pi}\int_{X_0} \kappa(s)ds$.
For the sake of brevity, we write the \emph{elastic energy} of the curve (see \cite{Gao-2023} and \cite{Singer-2008}) as the integral
\begin{eqnarray*}
E(X_0):= \int_0^{L_0} (\kappa_0(s))^2 ds.
\end{eqnarray*}
Now we consider a curvature evolution problem for locally convex curves, namely
\begin{equation}\label{eq:1.2.202302}
\left\{
\begin{aligned}
\frac{\partial X}{\partial t}(\theta, t)=\left(2w_k(\theta, t)-\rho_k(\theta, t)+f(t)\right)N(\theta, t),&\\
X(\theta, 0)=X_0(\theta), \quad  (\theta, t)\in [0, 2m\pi] \times [0, T_{\max}),
\end{aligned}
\right.
\end{equation}
where
\begin{eqnarray} \label{eq:1.3.202302}
w_k(\theta)= \sum_{i=0}^{k-1} p\left(\theta+\frac{2im\pi}{k}\right)~ \text{and} ~\rho_k(\theta)= \sum_{i=0}^{k-1} \rho\left(\theta+\frac{2im\pi}{k}\right),
\end{eqnarray}
and the nonlocal term is defined by
\begin{eqnarray} \label{eq:1.4.202302}
f(t)=\frac{\int_0^{2m\pi}\kappa^2\frac{\partial^2\rho_k}{\partial \theta^2}\text{d}\theta
-\int_0^{2m\pi}\kappa^2\rho_k\text{d}\theta}{\int_0^{2m\pi}\kappa^2\text{d}\theta}.
\end{eqnarray}
Our main theorem is the following statement.
\begin{theorem}\label{thm:1.1.202302}
Let $X_0$ be a smooth, closed and locally convex planar curve. The flow \eqref{eq:1.2.202302} has a global solution and
keeps both the local convexity and the elastic energy of the evolving curve. As time goes to infinity, the curve $X(\cdot, t)$
converges smoothly to a locally convex curve of constant $k$-order width.
In particular, the limiting curve is a multiple circle if and only if the initial curve is $k$-symmetric.
\end{theorem}

Since some locally convex curves appear as self-similar solutions \cite{Abresch-Langer-1986, Halldorsson-2012} to the classical Curve Shortening Flow,
it is quite natural to consider curvature flows for these curves.
During the last years, Xiaoliu Wang and his collaborators did some important research on this subject, see \cite{Wang-Kong-2014, Wang-Li-Chao-2017, Wang-Wo-Yang-2018}.
For more theories and applications of curvature flows of curves, one should also consult the monograph \cite{Chou-Zhu-2001} and suitable references therein.

\begin{remark}\label{remk:1.2.202302}
Comparing with models in the papers by Gao-Pan \cite{Gao2014}, Gao-Zhang \cite{Gao-Zhang-2017} and Fang \cite{Fang2017},
a complicated nonlocal term $f(t)$ is used in the flow \eqref{eq:1.2.202302} with the aim to preserve the elastic energy of $X(\cdot, t)$.
This property guarantees the global existence of the flow. This term is motivated by the first author's recent work \cite{Gao-2023}, where he
introduces a new curvature flow to answer Yau's problem of evolving one curve to another in the case of locally convex curves.
\end{remark}

\begin{remark}\label{remk:1.3.202302}
The original goal of this paper was to understand convex domains of (k-order) constant width via curvature flows.
In fact, convex curves (or convex domains) of constant width and higher dimensional analogues are of special interest in geometry.
As far as we know, the famous related Blaschke-Lebesgue problem \cite{Anciaux-Georgiou-2009, Anciaux-Guilfoyle-2011, B-L-O} for dimension $n\geq 3$ is still open.
One may consult the monograph \cite{Martini-Montejano-Oliveros} for more results on related topics.
\end{remark}

This paper is organized as follows. In Section 2, short-time existence of the flow \eqref{eq:1.2.202302} is proved. In Section 3, global existence
is obtained. And in Section 4, we prove convergence and the main theorem.

%
%
%

\section{Short-time existence}\label{s2}
Suppose $X : S^1\times [0,T)\rightarrow \mathbb{E}^2$ is a family of smooth, closed and locally convex curves in the plane evolving according
to the flow \eqref{eq:1.2.202302}. Usually, the tangent angle $\theta=\theta(s, t)$ varies as time goes.
As experts did in previous studies (see Proposition 1.1 in the paper \cite{Chou-Zhu-2001}), we consider the next flow instead of \eqref{eq:1.2.202302} such that $\theta$ is a variable independent of time $t$:
\begin{equation}\label{eq:2.1.202302}
\left\{
\begin{aligned}
&\frac{\partial \tilde{X}}{\partial t}=\alpha (\theta, t)T(\theta, t)+\left(2\omega_k(\theta, t)-\rho_k(\theta, t)+f(t)\right)N(\theta, t),
\\
&\tilde{X}(\theta, 0)=X_0(\theta), \quad  (\theta, t)\in [0, 2m\pi] \times [0, T_{\max}),
\end{aligned}
\right.
\end{equation}
where $\alpha$ is given by
\begin{eqnarray*}
\alpha=-2\frac{\partial w_k}{\partial \theta}+\frac{\partial \rho_k}{\partial \theta}.
\end{eqnarray*}
It follows from Proposition 1.1 in \cite{Chou-Zhu-2001} that the solutions to \eqref{eq:2.1.202302} and \eqref{eq:1.2.202302} are the same except altering the parametrization.
So the short-time existence of the flow \eqref{eq:1.2.202302} is equivalent to that of \eqref{eq:2.1.202302}.

Both the equations \eqref{eq:1.2.202302} and \eqref{eq:2.1.202302} are fully non-linear parabolic equation systems. The main idea of the proof for short-time existence is to reduce
these complicated equations to a semi-linear system of the evolution equation of $p$ and $\rho_k$.

Since the Frenet frame can be expressed as
\begin{eqnarray*}
T=(\cos\theta, \sin\theta), \quad N=(-\sin\theta, \cos\theta),
\end{eqnarray*}
one gets the \emph{Frenet formulae}
\begin{eqnarray*}
\frac{\partial T}{\partial\theta}=N, \quad \frac{\partial N}{\partial\theta}=-T.
\end{eqnarray*}
Set $\beta(\theta, t) = 2w_k(\theta, t)-\rho_k(\theta, t)+f(t)$.
Applying the equations (1.14)-(1.17) in the book \cite{Chou-Zhu-2001}, one obtains from \eqref{eq:2.1.202302},
\begin{eqnarray}
&& \frac{\partial T}{\partial t} = \left(\alpha \kappa + \frac{\partial \beta}{\partial s} \right)N
     = \left(\alpha + \frac{\partial \beta}{\partial \theta} \right) \kappa N, \label{eq:2.2.202302}
\\
&& \frac{\partial N}{\partial t} = -\left(\alpha \kappa + \frac{\partial \beta}{\partial s} \right)T
     = -\left(\alpha + \frac{\partial \beta}{\partial \theta} \right) \kappa T, \label{eq:2.3.202302}
\\
&& \frac{\partial \theta}{\partial t} = \alpha \kappa + \frac{\partial \beta}{\partial s}
     = \left(\alpha + \frac{\partial \beta}{\partial \theta} \right) \kappa, \label{eq:2.4.202302}
\\
&& \frac{\partial \kappa}{\partial t} = \kappa^2\left(\frac{\partial^2 \beta}{\partial \theta^2} + \beta\right). \label{eq:2.5.202302}
\end{eqnarray}
By the choice of $\alpha$, both the Frenet frame $\{T, N\}$ and the tangent angle $\theta$ are independent of the time:
\begin{eqnarray}\label{eq:2.6.202302}
\frac{\partial T}{\partial t} \equiv 0,~~ \frac{\partial N}{\partial t} \equiv 0,~~ \frac{\partial \theta}{\partial t} \equiv 0.
\end{eqnarray}
So the support function satisfies
\begin{eqnarray*}
\frac{\partial p}{\partial t}=-\frac{\partial }{\partial t}\left\langle X, N\right\rangle = -(2w_k-\rho_k+f(t))=\rho_k-2w_k-f(t).
\end{eqnarray*}
Since
\begin{eqnarray*}
\frac{\partial p}{\partial\theta}=-\left\langle \frac{\partial X}{\partial\theta}, N\right\rangle -\left\langle X, \frac{\partial N}{\partial\theta}\right\rangle=\langle X, T\rangle,
\end{eqnarray*}
we have
\begin{eqnarray*}
\frac{\partial^2 p}{\partial\theta^2} = \left\langle \frac{\partial X}{\partial s} \frac{\partial s}{\partial\theta}, T\right\rangle = \rho -p.
\end{eqnarray*}
So one obtains
\begin{equation}\label{eq:2.7.202302}
\rho =\frac{\partial^2 p}{\partial\theta^2}+p
\end{equation}
and
\begin{equation}\label{eq:2.8.202302}
\rho_k =\frac{\partial^2 w_k}{\partial\theta^2}+ w_k.
\end{equation}
Thus, the radius of curvature satisfies
\begin{eqnarray}
\frac{\partial \rho}{\partial t}&=&\frac{\partial }{\partial t}(p+p_{\theta\theta})
\nonumber\\
&=&\frac{\partial p}{\partial t}+\frac{\partial^2}{\partial\theta^2}\left(\frac{\partial p}{\partial t}\right)
\nonumber\\
&=&\rho_k-2w_k-f(t)+\frac{\partial^2}{\partial\theta^2}\left( \rho_k-2w_k-f(t)  \right)
\nonumber\\
&=&\frac{\partial^2\rho_k}{\partial\theta^2}-\rho_k-f(t), \label{eq:2.9.202302}
\end{eqnarray}
and one also has the evolution equation of the $k$-width function:
\begin{equation}\label{eq:2.10.202302}
\frac{\partial w_k}{\partial t}=-k(2w_k-\rho_k+f(t)) =k\left(\frac{\partial^2 w_k}{\partial \theta^2}-w_k-f(t)\right).
\end{equation}
Combining (\ref{eq:2.8.202302}) and (\ref{eq:2.10.202302}), one immediately obtains the evolution equation of $\rho_k$:
\begin{eqnarray}\label{eq:2.11.202302}
\frac{\partial \rho_k}{\partial t}(\theta, t)&=& k\left(\frac{\partial^2 \rho_k}{\partial\theta^2}(\theta, t)-\rho_k(\theta, t)-f(t)\right).
\end{eqnarray}

In the evolution equation of $\rho_k$, the term $f(t)$ contains the function $\rho$. One could not solve the the evolution equation of $\rho_k$ directly.
In order to get the \emph{short-time existence} of the flow, one needs to consider the above equations as a system.

\vskip.3cm
\begin{lemma}\label{lem:2.1.201909}
The nonlinear problem (\ref{eq:2.1.202302}) is equivalent to the following system on the domain $[0, 2m\pi]\times [0, T_{\max})$,
\begin{equation}\label{eq:2.12.202302}
\left\{
\begin{aligned}
&\frac{\partial p}{\partial t}(\theta, t)= \rho_k(\theta, t)-2w_k(\theta, t)-f(t),
\\
&\frac{\partial \rho_k}{\partial t}(\theta, t)= k\left(\frac{\partial^2 \rho_k}{\partial\theta^2}(\theta, t)-\rho_k(\theta, t)-f(t)\right),
\\
&\frac{\partial w_k}{\partial t} (\theta, t)= k\left(\frac{\partial^2 w_k}{\partial \theta^2}(\theta, t)-w_k(\theta, t)-f(t)\right),
\\
& \frac{\partial \rho}{\partial t} (\theta, t) = \frac{\partial^2\rho_k}{\partial\theta^2}(\theta, t)-\rho_k(\theta, t)-f(t),
\end{aligned}
\right.
\end{equation}
with initial values for $\theta\in [0, 2m\pi]$,
\begin{eqnarray*}
p(\theta, 0) = p_0(\theta), ~~w_k(\theta, 0)= w_{k0}(\theta), ~~\rho_k(\theta, 0) = \rho_{k0}(\theta), ~~\rho(\theta, 0) = \rho_{0}(\theta).
\end{eqnarray*}

\end{lemma}
\begin{proof}
If $X(\cdot, t)$ is a family of locally convex curves evolving according to (\ref{eq:2.1.202302}), we immediately have evolution equations in (\ref{eq:2.12.202302}).
Suppose (\ref{eq:2.12.202302}) has smooth and positive solutions. Then one may construct a family of locally convex curves by $p$ according to
\begin{eqnarray}\label{eq:2.13.202302}
X(\theta, t)=\frac{\partial p}{\partial\theta}(\theta, t)T(\theta)-p(\theta, t)N(\theta),
\end{eqnarray}
where $T(\theta)$ and $N(\theta)$, parameterized by the tangent angle $\theta$, form the Frenet frame of the curve at every point $X(\theta, t)$.
Therefore, the curve $X(\cdot, t)$ satisfies
\begin{eqnarray*}
\frac{\partial X}{\partial t}=\frac{\partial^2 p}{\partial t \partial \theta}T-\frac{\partial p}{\partial t}N
=\frac{\partial}{\partial\theta}\left(\frac{\partial p}{\partial t}\right) T-\frac{\partial p}{\partial t}N =
\alpha T+\left(2w_k-\rho_k+f\right)N.
\end{eqnarray*}
This is the evolution equation in \eqref{eq:2.1.202302}. Thus we are done.
\end{proof}

\begin{lemma}\label{lem:2.2.201909}
The flow \eqref{eq:2.1.202302} has a unique and smooth solution on some time interval.
\end{lemma}
\begin{proof}
According to Lemma \ref{lem:2.1.201909}, one needs show the system \eqref{eq:2.12.202302} has positive and smooth solutions on some time interval.
Denote by $\eta_k(\theta, t) := \frac{\partial\rho_k}{\partial \theta}(\theta, t)$. Then the function $\eta_k$ satisfies a linear equation
\begin{eqnarray}\label{eq:2.14.202302}
\frac{\partial \eta_k}{\partial t}(\theta, t)= k\left(\frac{\partial^2 \eta_k}{\partial\theta^2}(\theta, t)-\eta_k(\theta, t)\right)
\end{eqnarray}
with a smooth initial value $\eta_k(\theta, 0) = \frac{\partial\rho_k}{\partial \theta}(\theta, 0)$. Solving the linear parabolic equation \eqref{eq:2.14.202302}
with the initial value, we get a smooth function
$\eta_k(\theta, t)$ on the domain $[0, 2m\pi]\times [0, +\infty)$. Since $\rho_k (\theta, t) - \int_0^{\theta}\eta_k(\widetilde{\theta}, t) d\widetilde{\theta}$ is
independent of $\theta$, there is a function $\lambda(t)$, to be determined, so that
\begin{eqnarray}\label{eq:2.15.202302}
\rho_k (\theta, t) = \int_0^{\theta}\eta_k(\widetilde{\theta}, t) d\widetilde{\theta} + \lambda(t).
\end{eqnarray}

By observing the system \eqref{eq:2.12.202302}, we may compute to obtain
$\frac{\partial}{\partial t} (\rho_k - k \rho) \equiv 0$. So $\rho_k (\theta, t)- k \rho(\theta, t)$ is independent of time $t$, i.e, we have
\begin{eqnarray}\label{eq:2.16.202302}
\rho_k (\theta, t)- k \rho(\theta, t) =\rho_k (\theta, 0)- k \rho(\theta, 0).
\end{eqnarray}
Substituting \eqref{eq:2.15.202302} into \eqref{eq:2.16.202302}, we have
\begin{eqnarray}\label{eq:2.17.202302}
\rho(\theta, t) = \frac{1}{k}\left[\int_0^{\theta}\eta_k(\widetilde{\theta}, t) d\widetilde{\theta} + \lambda(t) -\rho_k (\theta, 0)+ k \rho(\theta, 0)\right].
\end{eqnarray}
Using the definition of $f(t)$ (see \eqref{eq:1.4.202302}) and the evolution equation of $\rho$, we get
\begin{eqnarray}\label{eq:2.18.202302}
\frac{d}{dt} \int_0^{2m\pi} \frac{1}{\rho(\theta, t)} = - \int_0^{2m\pi} \frac{1}{\rho^2(\theta, t)}\frac{\partial \rho}{\partial t} d \theta \equiv 0.
\end{eqnarray}
Therefore, the function $\lambda(t)$ is uniquely determined by the indentity
\begin{eqnarray}\label{eq:2.19.202302}
k\int_0^{2m\pi} \frac{d\theta }{\int_0^{\theta}\eta_k(\widetilde{\theta}, t) d\widetilde{\theta} + \lambda(t) -\rho_k (\theta, 0)+ k \rho(\theta, 0)}
\equiv \int_0^{2m\pi} \frac{1}{\rho_0 (\theta)} d\theta.
\end{eqnarray}

Once we have the function $\lambda(t)$, we get the values of $\rho_k (\theta, t)$, $\rho (\theta, t)$ and $f(t)$. So, integrating the evolution equations,
we obtain $w_k(\theta, t)$ and $p(\theta, t)$, respectively. By the continuity of $\rho_k (\theta, t)$, $\rho (\theta, t)$, $w_k(\theta, t)$ and $p(\theta, t)$,
the system \eqref{eq:2.12.202302} has positive and smooth solutions on some small time interval.
\end{proof}

\begin{remark}\label{remk:2.3.202302}
The equation \eqref{eq:2.16.202302} says that $\rho$ and $\rho_k$ have a concise relation.
The support function $p$ and the width function $w_k$ have a similar relation as shown in the equation \eqref{eq:2.16.202302}. This fact will be
used in the proof of Theorem \ref{thm:4.5.202302}.
\end{remark}

Lin and Tsai \cite{Lin-Tsai-2009} have considered a relative linear equation (compare to the system \eqref{eq:2.12.202302}) which can be used to answer Yau's problem of evolving one curve to another.
Recent progress on this problem can be found in the papers \cite{Gao-Zhang-2019, Gao-2023, McCoy-Schrader-Wheeler}.

\section{Long term existence}\label{s3}
In this section, we prove that the flow (\ref{eq:2.1.202302}) exists on the time interval $[0, +\infty)$. The main idea is to show that the radius of curvature $\rho$
has both uniformly positive lower and upper bounds. If so, then the flow can be infinitely extended, and in the evolution process the evolving curve $X(\cdot, t)$ is smooth
and locally convex.
Let $f(\theta, t)$ be a continuous function defined on $[0, 2m\pi] \times [0, T_{\max})$. We define
\begin{eqnarray*}
f_{\max}(t) = \max \{f(\theta, t)| \theta \in [0, 2m\pi]\}, ~~f_{\min}(t) = \min \{f(\theta, t)| \theta \in [0, 2m\pi]\}.
\end{eqnarray*}

\begin{lemma}\label{lem:3.1.202302}
Every order derivative of $\rho$ with respect to $\theta$ has uniform bounds if the flow (\ref{eq:2.1.202302}) preserves the local convexity of the evolving curve.
\end{lemma}
\begin{proof}
If the evolving curve $X(\cdot, t)$ is locally convex under the flow (\ref{eq:2.1.202302}), then we have the evolution equation of $\rho_k$ as (\ref{eq:2.12.202302}).
Differentiating this equation with respect to $\theta$ gives
\begin{eqnarray*}
\frac{\partial^2 \rho_k}{\partial t \partial \theta} = k\frac{\partial^3\rho_k }{\partial \theta^3}-k\frac{\partial\rho_k}{\partial\theta}.
\end{eqnarray*}
Let $u(\theta, t)=\frac{1}{2}\mid\frac{\partial\rho_k}{\partial\theta}\mid^2$. Then this function satisfies
\begin{eqnarray*}
\frac{\partial u}{\partial t}=k\left(\frac{\partial^2 u}{\partial \theta^2}-\left(\frac{\partial^2\rho_k}{\partial\theta^2}\right)^2\right)-2ku.
\end{eqnarray*}
Set $v(\theta, t)=e^{2kt}u(\theta, t)$. Then $v(\theta, 0)=u(\theta, 0)$ and
\begin{eqnarray*}
\frac{\partial v}{\partial t} \leq k\frac{\partial^2 v}{\partial \theta^2}.
\end{eqnarray*}
Applying the maximum principle, one obtains
$v_{\max}(t)\leq v_{\max}(0)$,
which implies
\begin{eqnarray*}
u(\theta, t)\leq u_{\max}(0)e^{-2kt}.
\end{eqnarray*}
Moreover,
\begin{eqnarray} \label{eq:3.1.202302}
\left|\frac{\partial\rho_k}{\partial\theta}(\theta, t) \right| \leq \max_\theta\left|\frac{\partial\rho_k}{\partial\theta}(\theta, 0)\right| e^{-kt}.
\end{eqnarray}
Denote by $C_i$ the constant $\max\limits_\theta\left|\frac{\partial^i \rho_k}{\partial\theta^i}(\theta, 0)\right|$, $i=2, 3, \cdots$.
Using the evolution equation of $\frac{\partial^i \rho}{\partial\theta^i}$, one may similarly prove that
$\left|\frac{\partial^i \rho_k}{\partial\theta^i}(\theta, t) \right|$ is bounded by $C_ie^{-kt}$.

Differentiating the evolution equation of $\rho$, one obtains
\begin{eqnarray*}
\frac{\partial}{\partial t}\left(\frac{\partial^i \rho}{\partial\theta^i}\right) =\frac{\partial^{i+2}\rho_k}{\partial\theta^{i+2}}-\frac{\partial^i \rho_k}{\partial\theta^i}.
\end{eqnarray*}
Since $\left|\frac{\partial^i \rho_k}{\partial\theta^i}(\theta, t) \right|$ decays exponentially, there exists a constant $M_i$, independent of time,
such that
\begin{eqnarray}\label{eq:3.2.202302}
\left|\frac{\partial^i \rho}{\partial\theta^i}(\theta, t) \right| \leq M_i, ~~(\theta, t)\in [0, 2m\pi]\times [0, T_{\max}),~~i=1, 2, \cdots.
\end{eqnarray}
The proof is finished.
\end{proof}

\begin{lemma}\label{lem:3.2.202302}
If the flow (\ref{eq:2.1.202302}) preserves the local convexity of the evolving curve, then the elastic energy is fixed as time goes.
\end{lemma}
\begin{proof}
Under the flow (\ref{eq:2.1.202302}), the curvature $\kappa(\theta, t)$ of the curve $X(\cdot, t)$ evolves according to (\ref{eq:2.5.202302}),
i.e.,  we have
\begin{eqnarray}\label{eq:3.3.202302}
\frac{\partial \kappa}{\partial t} = \kappa^2\left(-\frac{\partial^2 \rho_k}{\partial \theta^2} + \rho_k(\theta, t)+f(t)\right).
\end{eqnarray}
So the elastic energy of the evolving curve satisfies
\begin{eqnarray*}
\frac{d E}{d t} (t) &=& \frac{d}{d t} \int_{X(\cdot, t)} \kappa^2(s, t) ds
\\
&=& \frac{d}{d t} \int_{mS^1} \kappa(\theta. t) d\theta
\\
&=& \int_{mS^1} \kappa^2 \left(-\frac{\partial^2 \rho_k}{\partial \theta^2} + \rho_k(\theta, t)+f(t) \right) d\theta.
\end{eqnarray*}
By the definition of $f(t)$, one has $\frac{d E}{d t} \equiv 0$.
\end{proof}

Since the elastic energy $E$ equals $\int_0^{2m\pi} \frac{1}{\rho(\theta, t)} \text{d}\theta$, one gets that under the flow (\ref{eq:2.1.202302})
\begin{eqnarray}\label{eq:3.4.202302}
\frac{2m\pi}{\rho_{\max}(t)}\leq E \leq\frac{2m\pi}{\rho_{\min}(t)},
\end{eqnarray}
if this flow preserves the local convexity of the evolving curve.
This observation together with the gradient estimate of $\rho$ lead to its uniform bounds.

\begin{lemma}\label{lem:3.3.202302}
Under the condition of Lemma \ref{lem:3.2.202302}, there exist two positive constants $m_0$ and $M_0$ independent of time such that the curvature radius is bounded as
\begin{eqnarray}\label{eq:3.5.202302}
m_0 \leq \rho(\theta, t) \leq M_0.
\end{eqnarray}
\end{lemma}
\begin{proof}
Under the flow (\ref{eq:2.1.202302}), the gradient estimate of $\rho$ tells us that
$\mid\frac{\partial\rho}{\partial\theta}\mid\leq M_1$, where $M_1$ is a positive constant, independent of $t$.
Fix the time $t$. By continuity of $\rho$, there exist $\theta_1$ and $\theta_2$ such that $\rho_{\min}(t) = \rho(\theta_1, t)$ and  $\rho_{\max}(t) = \rho(\theta_2, t)$.
So
\begin{eqnarray*}
\ln\rho_{\max}(t)-\ln\rho_{\min}(t) = \int_{\theta_1}^{\theta_2}\frac{1}{\rho} \frac{\partial\rho}{\partial\theta} \text{d}\theta
\leq \int_0^{2m\pi}\frac{1}{\rho}\left|\frac{\partial\rho}{\partial\theta}\right| \text{d}\theta \leq M_1E.
\end{eqnarray*}
Therefore,
\begin{equation}\label{eq:3.6.202302}
\rho_{\max}(t)\leq\rho_{\min}(t)e^{M_1E}.
\end{equation}
Setting $m_0 = \frac{2m\pi}{E}e^{-M_1E}$ and $M_0 = \frac{2m\pi}{E}e^{M_1E}$, and combining (\ref{eq:3.4.202302}) and (\ref{eq:3.6.202302}),
one has the estimate (\ref{eq:3.5.202302}).
\end{proof}

Using this lemma, we may show that the flow (\ref{eq:2.1.202302}) preserves the local convexity of the evolving curve.

\begin{lemma}\label{lem:3.4.202302}
If the initial curve $X_0$ is locally convex, then the evolving curve $X(\cdot, t)$ is also locally convex under the flow (\ref{eq:2.1.202302}).
\end{lemma}
\begin{proof}
Suppose the flow exists on time interval $[0, T_{\max})$ and there is a positive $t_0<T_{\max}$ such that $X(\cdot, t)$ is locally convex on time interval $[0, t_0)$
but the minimum of the curvature $\kappa(\theta, t_0)$, with respect to $\theta$, is 0.

By the proof of Lemma \ref{lem:3.3.202302}, the curvature has a lower bound $\kappa(\theta, t) \geq \frac{E}{2m\pi}e^{-M_1E}$ under the flow (\ref{eq:2.1.202302})
for every $(\theta, t)\in [0, 2m\pi]\times [0, t_0)$. The continuity of curvature implies that $\kappa(\theta, t_0) \geq \frac{E}{2m\pi}e^{-M_1E} >0$ holds for all $\theta$.
A contradiction.
\end{proof}

\begin{theorem}\label{thm:3.5.202302}
If the initial curve $X_0(\theta)$ is locally convex, then the flow (\ref{eq:2.1.202302}) has a unique smooth solution $X(\cdot, t)$ on $[0, 2m\pi]\times [0,+\infty)$.
\end{theorem}
\begin{proof}
Suppose the flow (\ref{eq:2.1.202302}) exists on the maximal time interval $[0, T_{\max})$ and $T_{\max}$ is a finite positive number. It follows from (\ref{eq:3.2.202302}) and
(\ref{eq:3.5.202302}) that $\kappa$ and all its derivatives are uniformly bounded on the time interval $[0, T_{\max})$.
So the nonlocal term $f(t)$ has uniform bound which is independent of $T_{\max}$.

By the evolution equation of the $k$-order width $w_k(\theta, t)$, its derivative $\frac{\partial^i w_k}{\partial \theta^i}$ satisfies
\begin{eqnarray*}
\frac{\partial}{\partial t}\left(\frac{\partial^i w_k}{\partial \theta^i}\right)
=k\frac{\partial^{i+2}w_k}{\partial\theta^{i+2}}-k\frac{\partial^i w_k}{\partial\theta^i}.
\end{eqnarray*}
Applying the same trick as in the proof of Lemma \ref{eq:3.1.202302}, one may show that $|\frac{\partial^i w_k}{\partial\theta^i}|^2$
decays exponentially, then $w_k(\theta, t)$ is also uniformly bounded on the time interval $[0, T_{\max})$.

Hence the velocity of the flow has uniform bound which is independent of $w$.
By the unique existence of the flow, one obtains a smooth and locally convex curve
\begin{eqnarray*}
X_{T_{\max}}(\theta) := X_0(\theta) + \int_0^{T_{\max}} \frac{\partial X}{\partial t} (\theta, t) dt.
\end{eqnarray*}
Let $X_{T_{\max}}(\theta)$ evolve according to the flow (\ref{eq:2.1.202302}). Then there exists a family of smooth, locally convex curves $X(\cdot, t)$ on the time interval
$[T_{\max}, T_{\max}+\varepsilon)$, where $\varepsilon$ is a positive number. By the unique existence of the flow (\ref{eq:2.1.202302}), this flow is extended on a larger time interval
$[0, T_{\max}+\varepsilon)$. This contradicts the maximality of $T_{\max}$.
\end{proof}

\section{Convergence}\label{s4}
In this section, we explore the asymptotic behavior of the flow (\ref{eq:2.1.202302}) and complete the proof of Theorem \ref{thm:1.1.202302}.

Let $X_0(\theta)$ be a locally convex plane curve with rotation number $m$ and tangent angle $\theta$,
where $\theta\in [0, 2m\pi]$. Expand the support function as
\begin{eqnarray*}
p(\theta)=\frac{a_0}{2}+\sum\limits_{n=1}^{\infty}\left(a_n\cos\frac{n\theta}{m}+b_n\sin\frac{n\theta}{m}\right),
\end{eqnarray*}
where the coefficients are expressed as
\begin{eqnarray*}
a_0= \frac{L_0}{m\pi},~~~
a_n=\frac{1}{m\pi}\int_{-m\pi}^{m\pi}p(\theta)\cos\frac{n\theta}{m} d\theta,
~~~
b_n=\frac{1}{m\pi}\int_{-m\pi}^{m\pi}p(\theta)\sin\frac{n\theta}{m} d\theta.
\end{eqnarray*}
So the $k$-order width of the curve is
\begin{eqnarray}\label{eq:4.1.202302}
w_k(\theta)=\frac{k a_0}{2}+\sum\limits_{n=1}^{\infty}a_n\sum\limits_{l=0}^{k-1} \cos\left(\frac{n\theta}{m}+\frac{2nl\pi}{k}\right)
     +\sum\limits_{n=1}^{\infty}b_n\sum\limits_{l=0}^{k-1}\sin\left(\frac{n\theta}{m}+\frac{2nl\pi}{k}\right).
\end{eqnarray}

For a positive integer $l$, one has the identities
\begin{eqnarray*}
\sin\left(\frac{2n\pi}{k}\right)+\sin\left(\frac{4n\pi}{k}\right)+\cdots+\sin\left(\frac{2(k-1)n\pi}{k}\right)
=0
\end{eqnarray*}
and
\begin{eqnarray*}
\cos\left(\frac{2n\pi}{k}\right)+\cos\left(\frac{4n\pi}{k}\right)+\cdots+\cos\left(\frac{2(k-1)n\pi}{k}\right)
=
\left\{
\begin{aligned}
-1, \quad n\neq kl, &\\
k-1, \quad n=kl.
\end{aligned}
\right.
\end{eqnarray*}
So one may compute
\begin{eqnarray}
w_k(\theta) &=& \frac{k a_0}{2} +\sum\limits_{n=1}^{\infty}a_n\cos \frac{n\theta}{m}\sum\limits_{l=0}^{k-1}\cos\frac{2nl\pi}{k}
-\sum\limits_{n=1}^{\infty}a_n\sin \frac{n\theta}{m}\sum\limits_{l=0}^{k-1}\sin\frac{2nl\pi}{k}
\nonumber\\
&& +\sum\limits_{n=1}^{\infty}b_n\sin \frac{n\theta}{m}\sum\limits_{l=0}^{k-1}\cos\frac{2nl\pi}{k}
   +\sum\limits_{n=1}^{\infty}b_n\cos \frac{n\theta}{m}\sum\limits_{l=0}^{k-1}\sin\frac{2nl\pi}{k}
\nonumber\\
&=& \frac{k a_0}{2}+\sum\limits_{n=1}^{\infty}k\left( a_{nk} \cos\left(\frac{nk\theta}{m}\right)+b_{nk}\sin\left(\frac{nk\theta}{m}\right)\right). \label{eq:4.2.202302}
\end{eqnarray}
Hence, one has the following proposition.

\begin{proposition}\label{pro:4.1.202302}
Let $X_0(\theta)$ be a locally convex plane curve with the rotation number $m$. If it is of constant $k$-order width, then
\begin{eqnarray*}
p(\theta)=\frac{a_0}{2}+\sum\limits_{n\neq kl}^{\infty}\left(a_n\cos\frac{n\theta}{m}+b_n\sin\frac{n\theta}{m}\right).
\end{eqnarray*}
\end{proposition}

\vskip.3cm
\begin{definition}\label{def:4.2.202302}
\textup{Let $X_0$ be a plane closed curve with the rotation number $m$.
If it is invariant under the rotation of the angle $\frac{2m\pi}{k}$,
then it is called $k$-\emph{symmetric}.}
\end{definition}

Moreover, we can prove the following proposition.
\begin{proposition}\label{pro:4.3.202302}
Let $X_0(\theta)$ be a locally convex curve with the rotation number $m$. If $X_0(\theta)$ is $k$-symmetric, then
\begin{eqnarray*}
p(\theta)=\frac{a_0}{2}+\sum\limits_{l=1}^{\infty}\left(a_{kl}\cos\frac{kl\theta}{m}+b_{kl}\sin\frac{kl\theta}{m}\right).
\end{eqnarray*}
\end{proposition}
\begin{proof}
By the Fourier expansion of the support function $p$, one obtains
\begin{equation*}
\begin{aligned}
p(\theta+\frac{2m\pi}{k})
=&\frac{a_0}{2}+\sum\limits_{n=1}^{\infty}\left(a_n\cos\left(\frac{n\theta}{m}+\frac{2n\pi}{k}\right)+b_n\sin\left(\frac{n\theta}{m}+\frac{2n\pi}{k}\right)\right)
\\
=&\frac{a_0}{2}+\sum\limits_{n=1}^{\infty}\cos\frac{n\theta}{m}\left(a_n\cos\frac{2n\pi}{k}+b_n\sin\frac{2n\pi}{k}\right)\\
+&\sum\limits_{n=1}^{\infty}\sin\frac{n\theta}{m}\left(b_n\cos\frac{2n\pi}{k}-a_n\sin\frac{2n\pi}{k}\right).\\
\end{aligned}
\end{equation*}
Since $X_0$ is $k$-symmetric, $p(\theta)=p(\theta+\frac{2m\pi}{k})$ holds for every $\theta\in [0, 2m\pi]$.
A comparism of the coefficients in the Fourier expansion of $p(\theta)$ and $p(\theta+\frac{2m\pi}{k})$ finishes the proof.
\end{proof}
Combining the propositions $\eqref{pro:4.1.202302}$ and $\eqref{pro:4.3.202302}$, one gets
\begin{proposition}\label{pro:4.4.202302}
Let $X_0$ be a $k$-symmetric, locally convex plane curve with the rotation number $m$. Then $X_0$ is of constant $k$-order width
if and only if it is an $m$-fold circle.
\end{proposition}

Now we turn to the proof of the remaining part of Theorem \ref{thm:1.1.202302}.

\begin{theorem}\label{thm:4.5.202302}
The evolving curve of the flow \eqref{eq:2.1.202302} converges to a locally convex curve of constant $k$-order width.
\end{theorem}
\begin{proof}
On one hand, from \eqref{eq:4.2.202302} we get the evolving equation of $w_k$, namely
\begin{equation}\label{eq:4.3.202302}
\frac{\partial w_k}{\partial t}(\theta, t)
=\frac{k}{2}a_0'(t)+k\sum\limits_{n=1}^{\infty}\left(a_{nk}'(t)\cos\left(\frac{nk}{m}\theta\right)
+b_{nk}'(t)\sin\left(\frac{nk}{m}\theta\right)\right)
\end{equation}
and
\begin{equation*}
\frac{\partial w_k}{\partial \theta}(\theta, t)
=k\sum\limits_{n=1}^{\infty}\left(-\frac{nk}{m}a_{nk}(t)\sin\left(\frac{nk}{m}\theta\right)
+\frac{nk}{m}b_{nk}(t)\cos\left(\frac{nk}{m}\theta\right)\right).
\end{equation*}
Moreover,
\begin{equation}\label{eq:4.4.202302}
\frac{\partial^2w_k}{\partial \theta^2}(\theta, t)
=-k\sum\limits_{n=1}^{\infty}\frac{n^2k^2}{m^2}\left(a_{nk}(t)\cos\left(\frac{nk}{m}\theta\right)
+b_{nk}(t)\sin\left(\frac{nk}{m}\theta\right)\right).
\end{equation}
Substituting \eqref{eq:4.4.202302} into the evolution equation of $w_k$ (see (\ref{eq:2.10.202302})), one gets
\begin{equation}
\begin{aligned}
\frac{\partial w_k}{\partial t}(\theta, t)
=-&\frac{k^2}{2}a_0(t)-kf(t)\\
-&\sum\limits_{n=1}^{\infty}\left(\left(\frac{n^2k^4}{m^2}+k^2\right)a_{nk}(t)\cos\left(\frac{nk\theta}{m}\right)
+\left(\frac{n^2k^4}{m^2}+k^2\right)b_{nk}(t)\sin\left(\frac{nk\theta}{m}\right)\right). \label{eq:4.5.202302}
\end{aligned}
\end{equation}
Comparing the coefficients of the right sides in \eqref{eq:4.5.202302} and \eqref{eq:4.3.202302}, we have
\begin{equation}\label{eq:4.6.202302}
\left\{
\begin{aligned}
a_0'(t)=&-ka_0(t)-2f(t), \\
a_{nk}'(t)=&-\left(\frac{n^2k^3}{m^2}+k\right)a_{nk}(t), \\
b_{nk}'(t)=&-\left(\frac{n^2k^3}{m^2}+k\right)b_{nk}(t). \\
\end{aligned}
\right.
\end{equation}
Integrating the last two equations in \eqref{eq:4.6.202302} yields
\begin{eqnarray*}
a_{nk}(t)=a_{nk}(0)e^{-\frac{n^2k^3+m^2k}{m^2}t},
\quad
b_{nk}(t)=b_{nk}(0)e^{-\frac{n^2k^3+m^2k}{m^2}t}.
\end{eqnarray*}
Hence
\begin{equation}\label{eq:4.7.202302}
w_k(\theta, t)=\frac{k}{2}a_0(t)+k\sum\limits_{n=1}^{\infty}\left( a_{nk}(0)
\cos\left(\frac{nk\theta}{m}\right)+b_{nk}(0)\sin\left(\frac{nk\theta}{m}\right)\right)e^{-\frac{n^2k^3+m^2k}{m^2}t}.
\end{equation}
By the evolution equation of $p(\theta, t)$ and $w_k(\theta, t)$, one gets
\begin{eqnarray}\label{eq:4.8.202302}
\frac{\partial p}{\partial t}(\theta, t) = \frac{1}{k}\frac{\partial w_k}{\partial t}(\theta, t),
\end{eqnarray}
which implies that
\begin{equation*}
\begin{aligned}
p(\theta, t)=&p(\theta, 0)+\frac{1}{k}\left(w_k(\theta, t)-w_k(\theta, 0)\right)\\
=&\frac{1}{2}a_0(0)+\sum\limits_{n=1}^{\infty}\left( a_{n}(0)\cos \frac{n\theta}{m}+ b_{n}(0)\sin \frac{n\theta}{m}\right)\\
&+\frac{1}{2}a_0(t)+\sum\limits_{n=1}^{\infty}\left( a_{nk}(0)\cos \frac{nk\theta}{m}+ b_{nk}(0)\sin \frac{nk\theta}{m}\right)e^{-\frac{n^2k^3+m^2k}{m^2}t}\\
&-\frac{1}{2}a_0(0)-\sum\limits_{n=1}^{\infty}\left( a_{nk}(0)\cos \frac{nk\theta}{m}+ b_{nk}(0)\sin \frac{nk\theta}{m}\right).
\end{aligned}
\end{equation*}
That is,
\begin{equation}
\begin{aligned}\label{eq:4.9.202302}
p(\theta, t)
=&\frac{1}{2}a_0(t)+\sum\limits_{n=1}^{\infty}\left( a_{nk}(0)\cos \frac{nk\theta}{m}+ b_{nk}(0)\sin \frac{nk\theta}{m}\right)e^{-\frac{n^2k^3+m^2k}{m^2}t}\\
&+\sum\limits_{n\neq kl}^{\infty}\left( a_{n}(0)\cos \frac{n\theta}{m}+ b_{n}(0)\sin \frac{n\theta}{m}\right).
\end{aligned}
\end{equation}

On the other hand, it follows from (\ref{eq:3.2.202302}) and (\ref{eq:3.5.202302}) that $\rho(\cdot, t)$ is uniformly bounded and equicontinuous.
According to the well-known Arzel\`{a}-Ascoli Theorem, the function $\rho(\cdot, t)$ has a convergent subsequence.
Suppose there are two convergent subsequences $\{\rho(\cdot, t_i)\}$ and $\{\rho(\cdot, t_j)\}$ such that
\begin{eqnarray*}
\lim_{t_i \rightarrow +\infty} \rho(\theta, t_i) = \widetilde{\rho}(\theta),  ~~\lim_{t_j \rightarrow +\infty} \rho(\theta, t_j) = \widetilde{\eta}(\theta),
\end{eqnarray*}
where $\widetilde{\rho}$ and $\widetilde{\eta}$ are two positive functions.

By (\ref{eq:4.9.202302}) and the identity (\ref{eq:2.7.202302}), the two functions $\widetilde{\rho}$ and $\widetilde{\eta}$
differ by a constant, i.e., $\widetilde{\rho}(\theta)= \widetilde{\eta}(\theta)+c_0$ holds for all $\theta \in [0, 2m\pi]$.
Since the flow (\ref{eq:2.1.202302}) preserves the elastic energy $\int_0^{2m\pi} \frac{1}{\rho(\theta, t)} d\theta$, one has
\begin{eqnarray} \label{eq:4.10.202302}
\int_0^{2m\pi} \frac{1}{\widetilde{\rho}(\theta)} d\theta = \int_0^{2m\pi} \frac{1}{\widetilde{\eta}(\theta)} d\theta.
\end{eqnarray}
So the constant $c_0$ has to be 0. The radius of curvature $\rho(\theta, t)$ converges to a limiting function as $t\rightarrow +\infty$.
Since also the function $\rho_k(\theta, t)$ converges, we set
\begin{eqnarray}\label{eq:4.11.202302}
\lim_{t \rightarrow +\infty} \rho_k(\theta, t) = \widetilde{\rho}_k(\theta),
\end{eqnarray}
where $\widetilde{\rho}_k$ is a positive function. Furthermore, the estimate (\ref{eq:3.1.202302}) tells us that this function is a constant function.
By the evolution equation of the $k$-order width $w_k(\theta, t)$, its derivative $\frac{\partial^i w_k}{\partial \theta^i}$ satisfies
\begin{eqnarray*}
\frac{\partial}{\partial t}\left(\frac{\partial^i w_k}{\partial \theta^i}\right)
=k\frac{\partial^{i+2}w_k}{\partial\theta^{i+2}}-k\frac{\partial^i w_k}{\partial\theta^i}.
\end{eqnarray*}
Applying the same trick as in the proof of Lemma \ref{eq:3.1.202302}, one may show that $|\frac{\partial^i w_k}{\partial\theta^i}|^2$
decays exponentially. So $w_k(\theta, t)$ also converges to a constant as $t\rightarrow +\infty$. Using the relation (\ref{eq:2.8.202302}), one has
\begin{eqnarray}\label{eq:4.12.202302}
\lim_{t \rightarrow +\infty} w_k(\theta, t) = \widetilde{\rho}_k.
\end{eqnarray}

The equation \eqref{eq:4.8.202302} shows that the support function and the width function have the relation
\begin{eqnarray}\label{eq:4.13.202302}
p(\theta, t)=p(\theta, 0)+\frac{1}{k}\left(w_k(\theta, t)-w_k(\theta, 0)\right),
\end{eqnarray}
the limit (\ref{eq:4.12.202302}) implies that $p(\theta, t)$ converges as $t \rightarrow +\infty$. The equation (\ref{eq:2.13.202302}) implies that the evolving curve
of the flow (\ref{eq:2.1.202302}) also converges to a curve $X_\infty$ as time goes to infinity. Finally, since the limit (\ref{eq:4.12.202302}) says that the $k$-order width function
converges to a constant, the limiting curve $X_\infty$ has constant $k$-order width.
\end{proof}

\begin{theorem}\label{thm:4.6.202302}
If the initial curve $X_0$ is a $k$-symmetric, locally convex closed plane curve
with the rotation number $m$, then the evolving curve $X(\cdot, t)$ under the flow (\ref{eq:2.1.202302}) converges
to an $m$-fold circle, and vice versa.
\end{theorem}
\begin{proof}
From \eqref{eq:4.13.202302}, we have that
$p(\theta, t)-\frac{1}{k}w_k(\theta, t)$
is constant independent of $t$, that is,
\begin{eqnarray*}
p(\theta, t)-\frac{1}{k}w_k(\theta, t)=p(\theta, 0)-\frac{1}{k}w_k(\theta, 0).
\end{eqnarray*}

Suppose that the initial curve $X_0$ is $k$-symmetric. From Proposition \ref{pro:4.3.202302}
and \eqref{eq:4.2.202302}, one gets
\begin{eqnarray*}
p(\theta, 0)-\frac{1}{k}w_k(\theta, 0)=0.
\end{eqnarray*}
Hence, $p(\theta, t)=\frac{1}{k}w_k(\theta, t)$,
which together with Theorem \ref{thm:4.5.202302} gives us that $\lim\limits_{t\rightarrow\infty}p(\theta, t)$ is constant,
that is, the limiting curve is an $m$-fold circle.

Conversely, if the flow \eqref{eq:2.1.202302} has a global solution on $[0, 2m\pi]\times [0, \infty)$
and the limiting curve is an $m$-fold circle with center $O$, then $\lim\limits_{t\rightarrow\infty}p(\theta, t)$ is a constant and \eqref{eq:4.9.202302} implies
\begin{eqnarray*}
a_{n}(0)=b_{n}(0)=0, n\neq kl.
\end{eqnarray*}
Therefore, we get
\begin{eqnarray*}
p(\theta, t)=\frac{1}{2}a_0(t)
+\sum\limits_{n=1}^{\infty}\left( a_{nk}(0)\cos \frac{nk\theta}{m}+ b_{nk}(0)\sin \frac{nk\theta}{m}\right)e^{-\frac{n^2k^3+m^2k}{m^2}t},
\end{eqnarray*}
which implies
\begin{eqnarray*}
p(\theta, 0)=p(\theta+\frac{2m\pi}{k}, 0)= \cdots =p(\theta+\frac{2m(k-1)\pi}{k}, 0).
\end{eqnarray*}
In this case, $X_0$ is $k$-symmetric with respect to the origin.
\end{proof}

The combination of Lemma \ref{lem:2.2.201909}, Theorem \ref{thm:3.5.202302}, Theorem \ref{thm:4.5.202302} and Theorem \ref{thm:4.6.202302}
yields the proof of the main result given in Theorem \ref{thm:1.1.202302}.

~\\
\textbf{Acknowledgments}
Deyan Zhang is supported by University Natural Science Research Project of Anhui Province (No. 2022AH040067).

{\bf Laiyuan Gao}

School of Mathematics and Statistics, Jiangsu Normal University

No.101, Shanghai Road, Xuzhou 221116, Jiangsu, P. R. China

Email: lygao@jsnu.edu.cn\\

{\bf Horst Martini}

Faculty of Mathematics, University of Technology Chemnitz

09107 Chemnitz, Germany

Email: horst.martini@mathematik.tu-chemnitz.de\\

{\bf Deyan Zhang}

School of Mathematical Sciences, Huaibei Normal University

No.100, Dongshan Road, Huaibei 235000, Anhui, P. R. China

Email: zhangdy8005@126.com

\end{document}